\documentclass[11pt]{article}
\usepackage{amsmath}
\usepackage{amssymb}
\usepackage{amsthm}
\usepackage{cite}
\usepackage{mathrsfs,geometry,xcolor}
\usepackage{kbordermatrix}
\usepackage{enumerate}
\usepackage{bm}
\usepackage{blkarray}
\usepackage{tikz}

\numberwithin{figure}{section}
\numberwithin{table}{section}
\numberwithin{equation}{section}
\makeatletter
\newcommand\xleftrightarrow[2][]{\ext@arrow 0099{\longleftrightarrowfill@}{#1}{#2}}
\def\longleftrightarrowfill@{\arrowfill@\leftarrow\relbar\rightarrow}
\makeatother

\usepackage{tikz}
\usetikzlibrary{arrows,shapes,quotes}
\usetikzlibrary{patterns,decorations.pathreplacing}

\usepackage{authblk}

\numberwithin{table}{section}
\numberwithin{equation}{section}

\theoremstyle{plain}
\newtheorem{theorem}{Theorem}%[section]
\newtheorem{proposition}{Proposition}%[section]
\newtheorem{corollary}{Corollary}%[section]

\theoremstyle{definition}
\newtheorem{definition}{Definition}%[section]

\newtheorem{remark}{Remark}%[section]

\usepackage{authblk}
\author[1,*]{ \textbf{Noel T. Fortun}}
\author[1]{\textbf{Angelyn R. Lao}}
\author[2]{\textbf{Luis F. Razon}}
 \author[1,3,4,5]{\textbf{Eduardo R. Mendoza}}

\affil[1]{\small \textit{Mathematics and Statistics Department, De La Salle University, Manila  0922, Philippines}}
\affil[2]{\textit{Chemical Engineering Department, De La Salle University, Manila  0922, Philippines }}
\affil[3]{\textit{Institute of Mathematical Sciences and Physics, University of the Philippines,  Los Ba\~{n}os, Laguna 4031, Philippines}}
\affil[4]{\textit{Max Planck Institute of Biochemistry, Martinsried near Munich, Germany}}
\affil[5]{\textit{Faculty of Physics, Ludwig Maximilian University, Munich 80539, Germany}}
\affil[*]{Corresponding author: \texttt{noel.fortun@dlsu.edu.ph}}

\title{\textbf{Robustness in power law kinetic systems with reactant-determined interactions}}
\date{}

\begin{document}

\maketitle
\begin{abstract}
Robustness against the presence of environmental disruptions can be observed in many systems of chemical reaction network. However, identifying the underlying components of a system that give rise to robustness is often elusive. The influential work of Shinar and Feinberg established simple yet subtle network-based conditions for absolute concentration robustness (ACR), a phenomenon in which a species in a mass-action system has the same concentration for any positive steady state the network may admit. In this contribution, we extend this result to embrace kinetic systems more general than mass-action systems, namely, power law kinetic systems with reactant-determined interactions (denoted by ``PL-RDK"). In PL-RDK, the kinetic order vectors of reactions with the same reactant complex are identical. As illustration, we considered a scenario in the pre-industrial state of global carbon cycle. A power law approximation of the dynamical system of this scenario is found to be dynamically equivalent to an ACR-possessing PL-RDK system. \\

\end{abstract}
\section{Introduction}
Robustness may be generally defined \cite{KITANO2004,SHINAR2010}
as a system-level dynamical property that allows a system to sustain its functions despite changes in internal and external conditions. This feature, in fact, is fundamental and ubiquitous in many biological processes, including cellular networks and entire organisms \cite{AEJ2014,KITANO2004,SHINAR2010}. One type of robust behavior is ``concentration robustness,'' wherein some quantity involving the concentrations of the different species in a network is fixed at equilibrium \cite{DEXTER2015}.  In a well-cited paper published in \textit{Science}, Shinar and  Feinberg \cite{SHINAR2010} introduced \textit{absolute concentration robustness} (ACR), a condition in which the concentration of a species in a network attains the same value in every positive steady state set by parameters and does not depend on  initial conditions.

Shinar and Feinberg presented sufficient structure-based conditions for a chemical reaction network (CRN) to display ACR on a particular species through a structural index called the \textit{deficiency}. This non-negative parameter has been the center of many powerful results in \textit{Chemical Reaction Network Theory} (CRNT), a theoretical body of work that associates the structure of a CRN to the  dynamical behaviour of the system \cite{FEIN1979,FEIN1995}. CRNT employs mathematical methods from graph theory, linear algebra, group theory and the theory of ordinary differential equations. In CRNT, chemical reaction networks are viewed as digraphs whose vertices (called \textit{complexes}) are mapped to non-negative vectors representing compositions of chemical species and whose arcs represent chemical reactions between them.  The Shinar-Feinberg Theorem on ACR holds for systems whose evolution are modelled by ordinary differential equations with mass-action kinetics (MAK), and is stated as follows:
\begin{quotation}
\noindent \textit{Consider a mass-action system that admits a positive steady state and suppose that the deficiency of the underlying reaction network is one. If there are two nonterminal nodes in the network that differ only in species $S$, then the system has absolute concentration robustness in $S$.}
\end{quotation}

Here, we show that this result extends to systems endowed with \textit{power law kinetics} (PLK), which generalize mass-action kinetics \cite{CLARKE1988,HORNJACK1972}. Several experiments have shown that the kinetic order of a reaction with respect to a given reactant is a function of the geometry within which the reaction occurs \cite{KOPELMAN1986,KOPELMAN1988,KOPELMAN1991,NEWHOUSE1988,SCHNELL2004}. 
In the case of reactions occurring within a three-dimensional homogenous space (as in mass-action systems), the kinetic order is the same as the number of molecules entering into the reaction. However, for systems characterized by molecular overcrowding (e.g., when other molecules deny the reactants from the supposedly allowable space, and to stickiness, when the reactants are found along the surfaces of the reaction vessel) the kinetic orders for the reactions can exhibit non-integer values  \cite{SAVAGEAU1998} found in power law formalism \cite{SAV1969,VOIT2000,VOIT2013}. For instance, in intracellular environments, which are highly structured and characterized by molecular crowding, reactions in vivo are likely to take place on membranes or channels and as such, reactions follow fractal-like kinetics \cite{BAJZER2008,CLEGG1992,KUTHAN2001,SCHNELL2004}. 
The presence of power law kinetics in reaction systems thus motivated CRN-based studies on PLK systems  (\cite{CORTEZ2018,FORTUNa,MURE2012,TAM2018} among others), some of which are extensions or modifications of existing results on MAK systems. 

This contribution specifically shows that the result of Shinar and Feinberg on ACR applies to a class of PLK system called\textit{ power law kinetic systems with reactant-determined interactions} (denoted by ``PL-RDK''). PL-RDK systems are kinetic systems with power law rate functions whose kinetic order vectors are identical for reactions with the same reactant complex. Since the kinetic orders of the mass-action rate functions are precisely the stoichiometric coefficients of the reactant complex, one can see that MAK is a special case of PL-RDK.

As an application, we employ the theorem to a power law approximation of the ODE system corresponding to a specific scenario in the  pre-industrial carbon cycle model developed by Anderies et al. \cite{ANDERIES}. Particulary, for the pre-industrial scenario where there are anthropogenic causes that reduce the capacity of terrestrial carbon pool to store carbon, the power law approximation leads to an ACR-possessing PL-RDK system.  

The rest of the paper is organized as follows: Section 2 assembles preliminary concepts in Chemical Reaction Network Theory required in stating and proving the results. Section 3 discusses the extension of the Shinar-Feinberg Theorem on ACR for PL-RDK systems. Section 4 applies the main result obtained from the previous section to a carbon cycle model.  In Section 5, we summarize our results and outline some research perspectives.

\section{Fundamentals of Chemical Reaction Networks and Kinetic Systems}

We recall some fundamental notions about chemical reaction networks (CRNs) and
chemical kinetic systems (CKS) assembled in \cite{AJMSM2015,TAM2018}. Some concepts introduced by Feinberg in \cite{FEIN1979,FEIN1995} are also reviewed. \\

\noindent \textbf{Notation.} We denote the real numbers by $\mathbb{R}$, the non-negative real numbers by $\mathbb{R}_{\geq0}$ and the positive real numbers by $\mathbb{R}_{>0}$.  Objects in the reaction systems are viewed as members of vector spaces. Suppose $\mathscr{I}$ is a finite index set. By $\mathbb{R}^\mathscr{I}$, we mean the usual vector space of real-valued functions with domain $\mathscr{I}$.  For $x \in \mathbb{R}^\mathscr{I}$, the $i^\text{th}$ coordinate of $x$ is denoted by $x_i$, where $i \in \mathscr{I}$. The sets $\mathbb{R}_{\geq 0}^\mathscr{I}$ and $\mathbb{R}_{>0}^\mathscr{I}$ are called the \textit{non-negative} and \textit{positive orthants} of $\mathbb{R}^\mathscr{I}$, respectively. Addition, subtraction, and scalar multiplication in $\mathbb{R}^\mathscr{I}$are defined in the usual way. If $x \in \mathbb{R}_{>0}^\mathscr{I}$ and $y \in \mathbb{R}^\mathscr{I}$, we define $x^y \in \mathbb{R}_{>0}$ by
$
x^y= \prod_{i \in \mathscr{I}} x_i^{y_i} .
$
The vector $\log x\in \mathbb{R}^\mathscr{I}$,where $x \in \mathbb{R}_{>0}^\mathscr{I}$, is given by 
$(\log x)_i = \log x_i,  \text{ for all } i \in \mathscr{I}.$ If $x,y \in  \mathbb{R}^\mathscr{I}$, the standard scalar product $x\cdot y \in  \mathbb{R}$ is defined by 
$
x \cdot y = \sum_{i \in \mathscr{I}} x_i y_i.
$ 
By the \textit{support} of $x \in \mathbb{R}^\mathscr{I}$, denoted by $\text{supp } x$, we mean the subset of $\mathscr{I}$ assigned with non-zero values by $x$. That is,
$
\text{supp } x := \{ i \in \mathscr{I} | x_i \neq 0 \}.
$ \\

\begin{definition}
A \textbf{chemical reaction network} (CRN) $\mathscr{N}$ is a triple $(\mathscr{S},\mathscr{C},\mathscr{R})$ of three finite sets:
\begin{enumerate}
\item a set $\mathscr{S}= \{X_1, X_2, \dots, X_m \}$ of \textbf{species};
\item a set $\mathscr{C} \subset \mathbb{R}^\mathscr{S}_{\geq 0}$ of \textbf{complexes};
\item a set $\mathscr{R} = \{R_1, R_2, \dots, R_r \}\subset \mathscr{C} \times \mathscr{C}$ of \textbf{reactions} such that $(y,y) \notin \mathscr{R}$ for any $y \in \mathscr{C}$, and  for each $y \in \mathscr{C}$, there exists $y' \in \mathscr{C}$ such that either $(y,y') \in \mathscr{R}$ or  $(y',y) \in \mathscr{R}$.
\end{enumerate} 
We denote the number of species with $m$, the number of complexes with $n$ and the number of reactions with $r$
\end{definition}

A CRN can be viewed as a digraph $(\mathscr{C},\mathscr{R})$ with vertex-labelling. In particular, it is a digraph where each vertex $y\in \mathscr{C}$ has positive degree and stoichiometry, i.e., there is a finite set $\mathscr{S}$ of species  such that $\mathscr{C}$ is a subset of $\mathbb{R}^{\mathscr{S}}_{\geq 0}$. The vertices are the complexes whose coordinates are in $\mathbb{R}^{\mathscr{S}}_{\geq 0}$, which  are the \textbf{stoichiometric coefficients}. The arcs are precisely the reactions.

We use the convention that an element $R_j = (y_j, y_j') \in \mathscr{R}$ is denoted by $R_j: y_j \rightarrow y_j' $. In this reaction, we say that $y_j$ is the \textbf{reactant} complex and $y'_j$ is the \textbf{product} complex. Connected components of a CRN are called \textbf{linkage classes}, strongly connected components are called \textbf{strong linkage classes}, and strongly connected components without outgoing arcs are called \textbf{terminal strong linkage classes}. We denote the number of linkage classes with $\ell$, that of the strong linkage classes with $s\ell$, and that of terminal strong linkage classes with $t$. A complex is called \textbf{terminal} if it belongs to a terminal strong linkage class; otherwise, the complex is called \textbf{nonterminal}. 

With each reaction $y\rightarrow y'$, we associate a \textbf{reaction vector} obtained by subtracting the reactant complex $y$ from the product complex $y'$. The \textbf{stoichiometric subspace} $S$ of a CRN is the linear subspace of $\mathbb{R}^\mathscr{S}$ defined by
$$S := \text{span }\{y' - y \in \mathbb{R}^\mathscr{S}| y\rightarrow y' \in \mathscr{R}\}.$$
The \textbf{rank} of the CRN, $s$, is defined as $s = \dim S$. 

Many features of CRNs can be examined by working in terms of finite dimensional spaces $\mathbb{R}^\mathscr{S}$ (species space), $\mathbb{R}^\mathscr{C}$ (complex space), and $\mathbb{R}^\mathscr{R}$ (reaction space). Suppose the set $\{ \omega_i \in \mathbb{R}^\mathscr{I} \mid i \in \mathscr{I} \}$ forms the \textit{standard basis} for $\mathbb{R}^\mathscr{I}$ where $\mathscr{I}=\mathscr{S,C}$ or $\mathscr{R}$. We recall four maps relevant in the study of CRNs: map of complexes, incidence map, stoichiometric map and Laplacian map. 
\begin{definition}
Let $\mathscr{N}=(\mathscr{S,C,R})$ be a CRN. 
\begin{enumerate}
\item The \textbf{map of complexes} $\displaystyle{Y: \mathbb{R}^\mathscr{C} \rightarrow \mathbb{R}^\mathscr{S}}$ maps the basis vector $\omega_y$ to the complex $ y \in \mathscr{C}$. 
\item The \textbf{incidence map} $\displaystyle{I_a : \mathbb{R}^\mathscr{R} \rightarrow \mathbb{R}^\mathscr{C}}$ is the linear map defined by mapping for each reaction $\displaystyle{R_j: y_j \rightarrow y_j' \in \mathscr{R}}$, the basis vector $\omega_j$ to the vector $\omega_{y_j'}-\omega_{y_j} \in \mathscr{C}$. 
\item The \textbf{stoichiometric map} $\displaystyle{N: \mathbb{R}^\mathscr{R} \rightarrow \mathbb{R}^\mathscr{S}}$ is defined as $N = Y \circ I_a$. 
\item For each $k \in \mathbb{R}^\mathscr{R}_{>0}$ , the linear transformation  $A_k : \mathbb{R}^\mathscr{C} \rightarrow \mathbb{R}^\mathscr{C}$ called \textbf{Laplacian map} is the mapping defined by 
$$A_k x:= \sum_{y \rightarrow y'\in\mathscr{R}}k_{y \rightarrow y'}x_y (\omega_{y'} -\omega_y),$$
where $x_y$ refers to the $y^\text{th}$ component of $x \in \mathbb{R}^\mathscr{C}$ relative to the standard basis. 
\end{enumerate}
\end{definition}

The following result, named as the \textit{Structure Theorem of the Laplacian Kernel} (STLK) by Arceo et al. in \cite{AJMSM2015}, is crucial in deriving  important results in CRNT \cite{FEIN1979,FEIN1995}. 

\begin{proposition}[Structure Theorem of the Laplacian Kernel (STLK), Prop. 4.1 \cite{FEIN1979}]\label{th: STLK}
Let $\mathscr{N}=\mathscr{(S,C,R)}$ be a CRN with terminal strong linkage classes $\mathscr{C}^1, \mathscr{C}^2, \dots, \mathscr{C}^t$. Let $k\in \mathbb{R}^\mathscr{R}_{>0}$ and $A_k$ its associated Laplacian. Then $\text{Ker } A_k$ has a basis $b^1,b^2,\dots,b^t$ such that $\text{supp }b^i = \mathscr{C}^i$ for all $i=1,2,\dots,t$.
\end{proposition} 

A non-negative integer, called the deficiency, can be associated to each CRN. The \textbf{deficiency} of a CRN, denoted by $\delta$, is the integer defined by $\delta = n - \ell - s$. This index has been the center of many studies in CRNT due to its relevance in the dynamic behaviour of the system. In \cite{FEIN1979}, Feinberg provided a geometric interpretation of deficiency:  $\delta = \dim (\text{Ker } Y \cap \text{Im } I_a)$. From this fact and the STLK, the following result follows. 
\begin{corollary}[Cor. 4.12 \cite{FEIN1979}]\label{cor:dimkerYAk}
Let $\mathscr{N}=(\mathscr{S,C,R})$ be a CRN with deficiency $\delta$ and $t$ terminal strong linkage classes. Then for each $k \in \mathbb{R}^\mathscr{R}_{>0}$,
$$
\dim (\text{Ker } Y A_k) \leq \delta + t.
$$
\end{corollary}

By \textit{kinetics} of a CRN, we mean the assignment of a rate function to each reaction in the CRN. It is defined formally as follows.
\begin{definition}
A \textbf{kinetics} of a CRN $\mathscr{N}=(\mathscr{S},\mathscr{C},\mathscr{R})$ is an assignment of a rate function $\displaystyle{K_{j}: \Omega_K \to \mathbb{R}_{\geq 0}}$ to each reaction $R_j \in \mathscr{R}$, where $\Omega_K$ is a set such that $\mathbb{R}^{\mathscr{S}}_{>0} \subseteq \Omega_K \subseteq {\mathbb{R}}^{\mathscr{S}}_{\geq 0}$. A kinetics for a network $\mathscr{N}$ is denoted by $$\displaystyle{K=[K_1,K_2,...,K_r]^\top:\Omega_K \to {\mathbb{R}}^{\mathscr{R}}_{\geq 0}}.$$ The pair $(\mathscr{N},K)$ is called the \textbf{chemical kinetic system (CKS)}.
\end{definition}

The above definition is adopted from \cite{WIUF2013}. It is expressed in a more general context than what one typically finds in CRNT literature. For power law kinetic systems, one sets $\Omega_K =\mathbb{R}^{\mathscr{S}}_{>0}$. Here, we focus on the kind of kinetics relevant to our context: 

\begin{definition}
A \textbf{chemical kinetics} is a kinetics $K$ satisfying the positivity condition: 
$$
\text{For each reaction } R_j: y_j \rightarrow y_j' \in \mathscr{R},  \text{ } K_{j}(c)>0 \text{ if and only if } \text{supp } y_j \subset\text{supp }c.
$$
\end{definition}

\noindent Once a kinetics is associated with a CRN, we can determine the rate at which the concentration of each species evolves at composition $c \in \mathbb{R}^\mathscr{S}_{>0}$. 

\begin{definition}
The \textbf{species formation rate function}  of a chemical kinetic system is the vector field  
$$f(c) = NK (c) = \displaystyle\sum_{y_j\rightarrow y'_j \in \mathscr{R}}K_j(c) (y_j'- y_j).$$
\noindent The equation $dc/dt=f(c)$ is the \textbf{ODE or dynamical system} of the CKS.  A \textbf{positive equilibrium or steady state} $c^*$ is an element of $\mathbb{R}^\mathscr{S}_{>0}$ for which $f(c^*) = 0$. The set of positive equilibria of a chemical kinetic system is denoted by $E_+(\mathscr{N}, K)$. 
\end{definition}

Power law kinetics is defined by an  $r \times m$ matrix $F=[F_{ij}]$, called the \textbf{kinetic order matrix}, and vector $k \in \mathbb{R}^\mathscr{R}$, called the \textbf{rate vector}.  

\begin{definition}
A kinetics $K: \mathbb{R}^\mathscr{S}_{>0} \rightarrow \mathbb{R}^\mathscr{R}$ is a \textbf{power law kinetics} (PLK) if
$$\displaystyle K_{i}(x)=k_i x^{F_{i,\cdot}} \quad \forall i=1,\dots,r$$
with $k_i \in \mathbb{R}_{>0}$ and $F_{ij} \in \mathbb{R}$.  A PLK system has \textbf{reactant-determined kinetics} (of type \textbf{PL-RDK}) if for any two reactions $R_i$, $R_j \in \mathscr{R}$ with identical reactant complexes, the corresponding rows of kinetic orders in $F$ are identical, i.e., $F_{ik}=F_{jk}$ for $k = 1,...,m$. 
\end{definition}

An example of PL-RDK is the well-known \textbf{mass-action kinetics} (MAK), where the kinetic order matrix is the transpose of the matrix representation of the map of complexes $Y$ \cite{FEIN1979}. That is, a kinetics is a MAK if
$$
K_{j}(c)=k_{j}x^{Y_{.,j}} \quad \text{for all } R_j: y_j \rightarrow y'_j \in \mathscr{R}
$$
where $k_{j} \in \mathbb{R}_{>0}$, called rate constants. Note that $Y_{.,j}$ pertains to the stoichiometric coefficients of a reactant complex $y_j \in \mathscr{C}$. 

\begin{remark}
In \cite{AJMSM2015}, Arceo et al. discussed several sets of kinetics of a network and drew a ``kinetic landscape". They identified two main sets: the \textit{complex factorizable} kinetics and its complement, the \textit{non-complex factorizable} kinetics. Complex factorizable kinetics generalize the key structural property of MAK -- that is, the species formation rate function decomposes as 
$$
\dfrac{dx}{dt}= Y \circ A_k \circ \Psi_k,
$$
where $Y$ is the map of complexes, $A_k$ is the Laplacian map, and $\Psi_k: \mathbb{R}^\mathscr{S}_{\geq 0} \rightarrow  \mathbb{R}^\mathscr{C}_{\geq 0}$ such that $I_a \circ K(x) = A_k \circ \Psi_k(x)$ for all $x \in \mathbb{R}^\mathscr{S}_{\geq 0}$. In the set of power law kinetics, PL-RDK is the subset of complex-factorizable kinetics.  
\end{remark}

We recall the definition of the $m \times n$ matrix $\widetilde{Y}$ from the work of M\"{u}ller and Regensburger \cite{MURE2012,MURE2014}: For each reactant complex, the associated column of $\widetilde{Y}$ is the transpose of the kinetic order matrix row of the complex's reaction, otherwise (i.e., for non-reactant complexes), the column is 0. We form the \textbf{$\bm{T}$-matrix} of a PL-RDK system by truncating away the columns of the non-reactant complexes in $\widetilde{Y}$, obtaining an $m \times n_r$ matrix, where $n_r$ denotes the number of reactant complexes \cite{TAM2018}.

\section{Absolute Concentration Robustness in PL-RDK Systems}

To illustrate absolute concentration robustness, we consider the following toy model:

\begin{equation}
\begin{tikzpicture}[baseline=(current  bounding  box.center)]
\tikzset{vertex/.style = {shape=circle,draw,minimum size=1.5em}}
\tikzset{edge/.style = {->,> = latex}}
% vertices
\node[vertex] (1) at  (0,0) {$X_1$};
\node[vertex] (2) at  (2,0) {$X_2$};
\node[draw=none,fill=none] (3) at (1,0) {};
%edges
\draw[edge]  (2.160) to (1.20);
\draw[edge]  (1.340) to (2.200);
\draw[edge,dashed]  (2,-0.47) to[out=-90,in=-90, looseness=1.5] (1,-0.17);
\end{tikzpicture}
\end{equation}
The map depicts a biochemical system involving transfer of material from two pools: $X_2$ to $X_1$ and $X_1$ to $X_2$, but with $X_2$ regulating the second process. Suppose the system evolves according to the following set of ODEs:
\begin{equation}\label{eq:ODE_toy}
\left.
  \begin{array}{rl}
\dot{X}_1 &= k_1X_2^{0.8} -k_2X_1^{0.5}X_2^{0.8} \\
\dot{X}_2 &= -k_1X_2^{0.8} + k_2X_1^{0.5}X_2^{0.8} 
\end{array}
 \right.
\end{equation}
The positive equilibrium of the system is attained when
\begin{equation}\label{eq:acr_toy}
X_1 = \left(\dfrac{k_1}{k_2}\right)^2 \quad \text{and} \quad
X_2 = \Gamma -\left(\dfrac{k_1}{k_2}\right)^2.
\end{equation}
where $\Gamma$ is the conserved amount of total material. These equations indicate that whenever $\Gamma>(k_1/k_2)^2$, a positive steady state exists. Furthermore, since $X_1$ has the same value in any steady state, the system exhibits ACR in $X_1$. 

We define absolute concentration robustness in PL-RDK systems as follows:

\begin{definition}\label{def:acr}
A PL-RDK system $(\mathscr{N},K)$ has \textbf{absolute concentration robustness(ACR) in species} $X_i \in \mathscr{S}$ if there exists $c^*\in E_+(\mathscr{N},K)$ and for every other $c^{**} \in E_+(\mathscr{N},K)$,  we have $c^{**}_i =c^*_{i}$.
\end{definition}

The following proposition adapts Theorem S3.15 found in supplementary online material of the paper of Shinar and Feinberg \cite{SHINAR2010} to deal with PL-RDK systems. 

\begin{proposition}\label{theorem:1}
Let $\mathscr{N}=(\mathscr{S,C,R})$ be a deficiency-one CRN. Suppose that $(\mathscr{N},K)$ is a PL-RDK system which admits a positive equilibrium $c^*$. If $y, y' \in \mathscr{C}$ are nonterminal complexes, then each positive equilibrium $c^{**}$ of the system satisfies the equation 
\begin{equation}\label{eq:theorem}
\left( T_{\cdot,y} - T_{\cdot,y'} \right) \cdot \log \left( \dfrac{c^{**}}{c^*} \right)=0.
\end{equation}
\end{proposition}

We largely reproduce the proof of Shinar and Feinberg in the said supplementary material of their paper. Since in their proof, the sums are often taken over all complexes, we use the notation of M\"{u}ller and Regensburger in \cite{MURE2012,MURE2014}:
$$
\widetilde{Y} = \left[ 
\begin{array}{c|c}
\text{ }T \text{ }& \text{ }  0 \text{ }\\
\end{array}
\right],
$$
adjoining $n-n_r$ zero columns for the non-reactant complexes, where $n_r$ denotes the number of reactant complexes. Furthermore, we write $\widetilde{y}$ for $\widetilde{Y}_{\cdot,y}$.

\begin{proof}
Assume that $c^*$ is a positive steady state of the PL-RDK system $(\mathscr{N},K)$. That is,
\begin{equation}\label{eq:c*}
\sum_{y \rightarrow y' \in \mathscr{R}} k_{y \rightarrow y'} (c^*)^{\widetilde{y}} (y'-y)=0.
\end{equation}
For each $y \rightarrow y' \in \mathscr{R}$, define the positive number $\kappa_{y \rightarrow y'}$ by
\begin{equation}\label{eq:kappa}
\kappa_{y \rightarrow y'}:=k_{y \rightarrow y'}(c^*)^{\widetilde{y}}.
\end{equation}
Thus, we obtain
\begin{equation}\label{eq:1}
\sum_{y \rightarrow y' \in \mathscr{R}} \kappa_{y \rightarrow y'} (y'-y)=0.
\end{equation}
Suppose that $c^{**}$ is also a positive equilibrium of the system. Hence,
\begin{equation}\label{eq:c**}
\sum_{y \rightarrow y' \in \mathscr{R}} k_{y \rightarrow y'} (c^{**})^{\widetilde{y}} (y'-y)=0.
\end{equation}
Define 
\begin{equation}\label{eq:mu}
\mu := \log c^{**} - \log c^*.
\end{equation}
With $\kappa \in \mathbb{R}^\mathscr{R}_{>0}$ given by Equation (\ref{eq:kappa}) and $\mu$ given by Equation (\ref{eq:mu}), it follows from Equation (\ref{eq:c**}) that 
\begin{equation}\label{eq:2}
\sum_{y \rightarrow y' \in \mathscr{R}} \kappa_{y \rightarrow y'} e^{\widetilde{y}\cdot \mu} (y'-y)=0.
\end{equation}
Let $\bm{1}^\mathscr{C} \in \mathbb{R}^\mathscr{C}$ such that 
$$
\bm{1}^\mathscr{C}  = \sum_{y \in \mathscr{C}} \omega_y.
$$
Observe that Equations (\ref{eq:1}) and  (\ref{eq:2}) can be respectively written as
$$
Y A_\kappa \bm{1}^\mathscr{C} =0, \text{ and }
Y A_\kappa \left( \sum_{y \in \mathscr{C}} e^{\widetilde{y}\cdot \mu} \omega_y \right) = 0.
$$
Equivalently, 
\begin{equation} \label{eq: w_c in Ker YAk}
\bm{1}^\mathscr{C} \in \text{Ker }YA_\kappa , \text{ and}
\end{equation}
\begin{equation} \label{eq: Sum in Ker YAk}
\sum_{y \in \mathscr{C}} e^{\widetilde{y}\cdot \mu}  \omega_y \in \text{Ker }YA_\kappa .
\end{equation}
Therefore, $c^*$ and $c^{**}$ are positive equilibria of the PL-RDK system $(\mathscr{N},K)$ if and only if Equations (\ref{eq: w_c in Ker YAk}) and (\ref{eq: Sum in Ker YAk}) hold. From Corollary \ref{cor:dimkerYAk}, we have 
\begin{equation}\label{eq:<=1+t}
\dim (\text{Ker } Y A_\kappa) \leq 1+t
\end{equation}
for the CRN under consideration. Let $\{ b^1,b^2,\dots,b^t \} \subset \mathbb{R}^\mathscr{C}_{\geq 0}$ be a basis for $\text{Ker }A_\kappa$ as in Proposition \ref{th: STLK} (STLK). Since $\text{Ker }A_\kappa \subseteq \text{Ker } YA_\kappa$, this basis of $\text{Ker }A_\kappa $ can be extended to form a basis of $\text{Ker } YA_\kappa$. Recall from Equation (\ref{eq: w_c in Ker YAk}) that $\bm{1}^\mathscr{C}$ is in $\text{Ker }YA_\kappa$. We assert that the set 
$
\{
\bm{1}^\mathscr{C}, b^1,b^2,\dots,b^t
\}
$
is a basis for $\text{Ker }YA_\kappa$ (and hence, equality holds in Equation (\ref{eq:<=1+t})). This follows if 
\begin{equation}\label{eq: not in Ker Ak}
\bm{1}^\mathscr{C} \notin \text{Span } \{b^1,b^2, \dots, b^t \}.
\end{equation}
From Proposition \ref{th: STLK}, every element of $\text{Ker }A_\kappa$ must have its support contained entirely in the set of terminal complexes. However, the support of $\bm{1}^\mathscr{C}$ consists of all complexes. By assumption, there are nonterminal complexes and hence, $\bm{1}^\mathscr{C}$ cannot lie in $\text{Ker }A_\kappa$ (i.e., Equation (\ref{eq: not in Ker Ak}) holds). \\

From Equation (\ref{eq: Sum in Ker YAk}), there exist scalars $\lambda_0,\lambda_1,\dots,\lambda_t$ such that 
\begin{equation}\label{eq:S27}
\sum_{y \in \mathscr{C}} e^{\widetilde{y}\cdot \mu}  \omega_y = \lambda_0 \bm{1}^\mathscr{C} + \sum_{i=1}^t \lambda_i b^i.
\end{equation}

\noindent Observe that each vector $b^i$, $i=0,1,\dots,t$, has its support entirely on terminal complexes. This fact, along with Equation (\ref{eq:S27}), implies that for each pair of nonterminal complexes $y \in \mathscr{C}$ and $y' \in \mathscr{C}$, we have
\begin{equation}\label{eq:S28}
\widetilde{y}\cdot\mu = \widetilde{y'}\cdot\mu .
\end{equation}
Since $y$ and $y'$ are nonterminal, they are reactant complexes. Hence, Equation (\ref{eq:S28}) may be written as
\begin{equation}
T_{\cdot,y}\cdot\mu = T_{\cdot,y'}\cdot\mu,
\end{equation}
which is equivalent to Equation (\ref{eq:theorem}) in Theorem \ref{theorem:1}.
\end{proof}

The extension of the Shinar-Feinberg Theorem on ACR to PL-RDK systems is stated as follows. 
\begin{theorem}
Let $\mathscr{N}=(\mathscr{S,C,R})$ be a deficiency-one CRN and suppose that $(\mathscr{N},K)$ is a PL-RDK system which admits a positive equilibrium.  If $y, y' \in \mathscr{C}$ are nonterminal complexes whose kinetic order vectors  differ only in species $X_i$, then the system has ACR in $X_i$.
\end{theorem}

\begin{proof}
Suppose $c^*$ and $c^{**}$ are positive equilibria of the PL-RDK system $(\mathscr{N},K)$. Observe that since $y, y' \in \mathscr{C}$ are nonterminal complexes whose kinetic order vectors differ only in species $X_i$, we have
$$
T_{\cdot,y} - T_{\cdot,y'} = aX_i 
$$
for some nonzero $a$. Thus Equation (\ref{eq:theorem}) reduces to 
$$
a(\log c^*_i - \log c^{**}_i)= 0.
$$
It follows that 
$$
c^*_i = c^{**}_i.
$$
That is, the system has ACR in species $X_i$.
\end{proof}

The ODE system in Equation (\ref{eq:ODE_toy}) can be translated into a dynamically equivalent CRN with associated kinetic order matrix by employing the notion of \textbf{total CRN representation of Generalized Mass Action (GMA) systems}, proposed by Arceo et al. \cite{AJMSM2015}. GMA system is a canonical framework used in Biochemical Systems Theory (BST) wherein every mass transfer rate is approximated separately with a power law term, and these terms are added together, with a plus sign for incoming fluxes and a minus sign for outgoing fluxes \cite{VOIT2000,VOIT2013}. For BST-related concepts, the reader may refer to the BST tutorial in the Appendix of Arceo et al. \cite{AJLM2017}. 

The total CRN representation of a GMA system allows for the CRN-based analysis of the dynamical system. Viewed as a GMA system, the set of ODEs in (\ref{eq:ODE_toy})  has the following total CRN representation:

\begin{equation}\label{eq:toy_CRN}
\left.
  \begin{array}{rrcl}
R_1 :& X_2 &\xrightarrow{k_1} & X_1 \\
R_2 :& X_1+X_2 &\xrightarrow{k_2} &2X_2
  \end{array}
 \right.
\end{equation}
with associated kinetic order matrix $F$ given by
$$ 
F=  
\kbordermatrix{
    & X_1 & X_2  \\
    R_1 & 0 & 0.8   \\
    R_2 & 0.5 & 0.8   \\
}.
$$
The CRN in (\ref{eq:toy_CRN}) is a deficiency-one network with nonterminal complexes $X_1+X_2$ and $X_2$ whose kinetic order rows differ only in $X_1$. The previous theorem indicates ACR in $X_1$, which agrees with the computation in (\ref{eq:acr_toy}). 

The following simple proposition provides some examples for the ACR theorem for PL-RDK systems. As preparation, we recall some notions from \cite{AJLM2018,TAM2018} which are used in the result. A PL-RDK is said to be \textbf{reactant set linear independent} (of type \textbf{PL-RLK}) if  the columns of $T$ are linearly independent. We also recall the \textbf{reactant matrix} $Y_\text{res}$, which is obtained from the matrix representation of $Y$ by removing the columns corresponding to non-reactant complexes.  Its image $\text{Im } Y_{\text{res}}$ is called the \textbf{reactant subspace} $R$, whose dimension $q$ is called the \textbf{reactant rank} of the CRN. The \textbf{reactant deficiency} $\delta_\rho$ is the difference between the number of reactant complexes $n_r$ and the reactant rank $q$.

\begin{proposition}
Let $(\mathscr{S,C,R})$ be a deficiency one reaction network, which with PL-RDK, admits a positive equilibrium. Suppose the network has zero reactant deficiency, two nonterminal complexes $y,y' \in \mathscr{C}$ differing only in $X_j$ and the map 
$$
\widehat{y}:= T \circ Y_\text{res}^{-1}: R  \rightarrow \text{Im } T
$$
is given by
$$
\widehat{y}(X_1,\dots,X_j,\dots, X_m) = (a_1X_1,\dots,a_jX_j\dots,a_mX_m), a_i \neq 0.
$$
Then the system is PL-RLK and has ACR in $X$.
\end{proposition}

\begin{proof}
Since $\widehat{y}$ is an isomorphism, $T=\widehat{y} \circ Y_\text{res}$ is also an isomorphism. This implies that the system is PL-RLK. The kinetic order vector difference of $y$ and $y'$ is $(0,\dots,ka_j,\dots,0)$ for some nonzero real $k$ so that Theorem 1's condition is fulfilled.
\end{proof}

\section{Application to a Carbon Cycle Model}

The pre-industrial carbon cycle model of Anderies et al. \cite{ANDERIES} is a simple mass balance which involves three interacting carbon pools: land, atmosphere and ocean. Pictorially, the system can be depicted using a biochemical map comprised of nodes that represent carbon pools, solid arrows that indicate transfer of carbon, and dashed arrows that indicate if a pool affects or modulates a process. Figure \ref{fig:ANPRIMap} presents the biochemical map of the model of interest. \\
\begin{figure}[!ht]
\centering
    \includegraphics[width=0.6\textwidth]{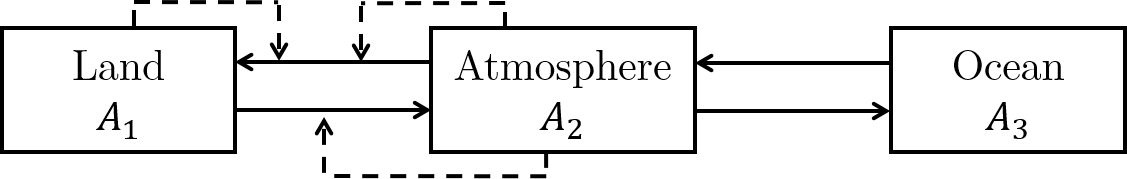}
      \caption{Biochemical map of the pre-industrial carbon cycle model of Anderies et al. \cite{ANDERIES}}  \label{fig:ANPRIMap}
\end{figure}

In our previous work \cite{FORTUNa}, we reviewed the model's design and underlying assumptions and described the parameters and ODEs present in the pre-industrial state of the carbon cycle model. We also approximated all rate processes by products of power law functions in order to obtain a GMA system approximation of the original system. The resulting ODEs of the approximation is given in (\ref{eq:GMA ODEs}):
\begin{equation}\label{eq:GMA ODEs}
\left.
  \begin{array}{cl}
\dot{A}_1 &= k_1 A_1^{p_1}A_2^{q_1} - k_2 A_1^{p_2}A_2^{q_2} \\[0.5em]
\dot{A}_2 &= k_2 A_1^{p_2}A_2^{q_2} - k_1 A_1^{p_1}A_2^{q_1} - a_m A_2 + a_m \beta A_3   \\[0.5em]
\dot{A}_3 &= a_m A_2 - a_m \beta A_3,
  \end{array}
 \right \}.
\end{equation} 

We also obtained in \cite{FORTUNa}, using total CRN representation of \cite{AJMSM2015}, the following deficiency-one CRN representation for the model: 
\begin{equation}
\left.
  \begin{array}{rcl}
A_1 + 2A_2 &\rightarrow &2A_1 + A_2 \\
A_1 + A_2 &\rightarrow & 2A_2 \\
A_2 &\rightleftarrows & A_3 \\
  \end{array}
 \right.
\end{equation}
Its associated kinetic order matrix is the transpose of the following $T$-matrix:
\begin{equation}
T= \begin{blockarray}{ccccl}
A_1 + 2A_2 & A_1 + A_2 & A_2 & A_3    \\
\begin{block}{[cccc]l}
 	p_1 & p_2 & 0 & 0 & A_1 \\
     q_1 & q_2 & 1 & 0 & A_2 \\
     0 & 0 & 0 & 1 &  A_3  \\
 \end{block}
\end{blockarray}.
\end{equation}

In the Appendix, it is shown that there is a scenario in the pre-industrial state leading to a GMA system approximation such that the kinetic order vectors of the nonterminal vertices $A_1 + 2A_2$ and $A_1 + A_2$ differ only in $A_2$; that is, $p_1- p_2=0$ and $q_1 - q_2 \neq 0$. In particular, this occurs when the human terrestrial carbon off-take term (which accounts for human activities that reduce the capacity of terrestrial pool to capture carbon such deforestation and land-use change) vanishes. Assuming the existence of a steady state, Theorem 1 indicates that the system has ACR in $A_2$. In fact, when $p_1=p_2$, steady state computation of the system in (\ref{eq:GMA ODEs}) yields the following equilibria set for the system:
$$
E_+(\mathscr{N},K)= 
\left\lbrace
\left[ 
\begin{array}{c}
A_1 \\
A_2 \\
A_3 \\  
\end{array} 
\right] \in \mathbb{R}^\mathscr{S}_{>0} \;\middle|\;  
\begin{array}{ll}
A_2 &=
\left(
\dfrac{k_2}{k_1} \right)^{\frac{1}{q_1-q_2}}, \\
A_3 &=\dfrac{1}{\beta}\left(
\dfrac{k_2}{k_1} \right)^{\frac{1}{q_1-q_2}}, \text{ and}\\
A_1 &= A_0 - \left( 1+ \dfrac{1}{\beta} \right) \left(
\dfrac{k_2}{k_1} \right)^{\frac{1}{q_1-q_2}}
\end{array}
\right\rbrace,
$$
where $A_0=$ total conserved carbon at pre-industrial state. 

\section{Conclusion and Outlook}

In conclusion, we summarize our results and outline some perspectives for further research.

\begin{enumerate}
\item We modified the Shinar-Feinberg Theorem on ACR for mass-action systems to include PL-RDK systems, a kinetic system more general than mass-action systems.
\item The theorem is applied to a power law approximation of Anderies et al.'s Earth's carbon cycle in its pre-industrial state. The analysis reveals that there is a scenario in the pre-industrial state which yields a power law approximation where there is ACR in the atmospheric carbon pool.  Specifically, the power law approximation leads to an ACR-possessing PL-RDK system when the human off-take coefficient, which accounts for the which accounts for human activities that reduce the capacity of terrestrial pool to sequester carbon, vanishes.
\item The investigation of other forms of ``concentration robustness" identified by Dexter et al. \cite{DEXTER2015} for PL-RDK systems offers a further interesting research perspective.
\item The extension of the stochastic analysis of CRNs with ACR of Anderson et al. \cite{AEJ2014} for PL-RDK systems is another promising area for further investigation. 
\end{enumerate}

\section*{Acknowledgements}
NTF acknowledges the support of the Department of Science and Technology-Science Education Institute (DOST-SEI), Philippines through the ASTHRDP Scholarship grant and Career Incentive Program (CIP). ARL and LFR held research fellowships from De La Salle University and would like to acknowledge the support of De La Salle University's Research Coordination Office. 
%
% ---- Bibliography ----
%
% BibTeX users should specify bibliography style 'splncs04'.
% References will then be sorted and formatted in the correct style.
%
%\bibliographystyle{splncs04}

%

\appendix
\section{Pre-industrial Carbon Cycle Model of Anderies et al.}\label{sec:A}

%Table \ref{tab:parameters} presents the parameters  used in the model with their respective values, which were taken verbatim from \cite{ANDERIES}. Furthermore, units and other details may be found in the Appendix of \cite{ANDERIES}. 
The complete set of ODEs for the pre-industrial state is given by
\begin{equation}\label{eq:ODE of preindustrial}
\left.
  \begin{array}{cl}
\dot{A}_1 &= r_{tc} [P(t)-R(t)] A_1 \left[ 1- \frac{A_1}{k} \right] - \alpha A_1 \\
\dot{A}_2 &= r_{tc} [R(t)-P(t)] A_1 \left[ 1- \frac{A_1}{k} \right] + \alpha A_1 - a_m A_2 + a_m \beta A_3 \\
\dot{A}_3 &= a_m A_2 - a_m \beta A_3 .
  \end{array}
 \right \}
\end{equation} 
where 
\begin{align*}
P(t) &= a_f A_2(t)^{b_f} \cdot \left[ a_p \cdot (a_T A_2(t) + b_T)^{b_p} \cdot e^{-c_p \cdot (a_T A_2(t) + b_T)} \right] \\
R(t)&=\left[ a_r \cdot (a_T A_2(t) + b_T)^{b_r} \cdot e^{-c_r \cdot (a_T A_2(t) + b_T)} \right] .
\end{align*}

For the description of the parameters, the reader is referred to \cite{ANDERIES} and the Appendix of \cite{FORTUNa}. The parameter values are identical to the values used in \cite{FORTUNa} but with $\alpha=0$. This particular parameter is assigned as the human terrestrial carbon off-take rate. It is associated to human activities such as clearing, burning or farming, which reduce the capacity of land to capture carbon. 

A power law approximation of the ODE system at an operating point is obtained to generate a Generalized Mass Action (GMA) System \cite{VOIT2000,VOIT2013}. Mathematically, GMA system approximation is equivalent to Taylor approximation up to the linear term in logarithmic space. The function $V(X_1,X_2,\dots,X_m)$ can be approximated by $\displaystyle{V = \alpha X_1^{p_1} X_2^{p_2} \cdots X_m^{p_m}}$ at an operating point where 
\begin{equation}\label{eq:formula}
p_i =\dfrac{\partial V}{\partial X_i} \cdot \dfrac{X_i}{V} \text{ and }
\alpha = V(X_1,X_2,\dots,X_m)X_1^{-p_1}X_2^{-p_2}\cdots X_m^{-p_m}.
\end{equation}

Table \ref{tab:power law approx} presents the four carbon fluxes present in the pre-industrial state of the Anderies et al. model, and their corresponding rate functions. Furthermore, the last column lists their respective target power law approximation. The last two functions, $a_m A_2$ and $a_m \beta A_3$, are already in the desired format and are thus, kept as is. To compute for the kinetic orders (and rate constants), we apply (\ref{eq:formula}). By taking the parameter values used in \cite{FORTUNa} but with $\alpha=0$, and assuming the initial values to be $A_1 =2850/4500$, $A_2=750/4500$ and $A_3=900/4500$ (as in\cite{ANDERIES}), the ODE system in (\ref{eq:ODE of preindustrial}) reaches the following steady state:
$
 A_1 = 0.7, \quad A_2 = 0.15\text{ and } A_3 = 0.15.
$  

\begin{table}[ht!]
\small
\begin{center}
\begin{tabular}{|lll|}
\hline
Carbon Flux & Function &  Power law approx. \\
\hline
$A_2 \rightarrow A_1$ & $K_1= r_{tc} P(t) A_1 \left[ 1- \frac{A_1}{k} \right]$ & $k_1 A_1^{p_1}A_2^{q_1}$  \\
$A_1 \rightarrow A_2$ & $K_2 = r_{tc} R(t) A_1 \left[ 1- \frac{A_1}{k} \right] + \alpha A_1$  & $k_2 A_1^{p_2}A_2^{q_2}$  \\
$A_2 \rightarrow A_3$ & $ K_3 = a_m A_2 $ & $ a_m A_2 $ \\
$A_3 \rightarrow A_2$ & $K_4 =a_m \beta A_3$ & $a_m \beta A_3$ \\
\hline
\end{tabular}
\caption{Power law approximation of the process rates.}\label{tab:power-law approx}
\end{center}
\end{table}

The algebraic calculations are implemented in Mathematica as shown in Figure \ref{fig:codes}. When $\alpha=0$ (i.e., the human off-take term vanishes), 
$$
p_1=p_2=\dfrac{2A_1-k}{A_1-k}.
$$
For the power law approximation, we choose values close to the equilibrium point as operating point: $A_1=0.69$, $A_2=0.155$ and $A_3=0.155$. Consequently, we obtain 
\begin{equation}\label{eq:p,q,k}
\left.
  \begin{array}{clcl}
		p_1 &=-68, &\quad   p_2 &=-68, \\
		q_1 &=0.580148, &\quad   q_2 &=0.910864.
  \end{array}
 \right.
\end{equation}
 \begin{figure}[!ht]
\centering
    \includegraphics[width=0.55\textwidth]{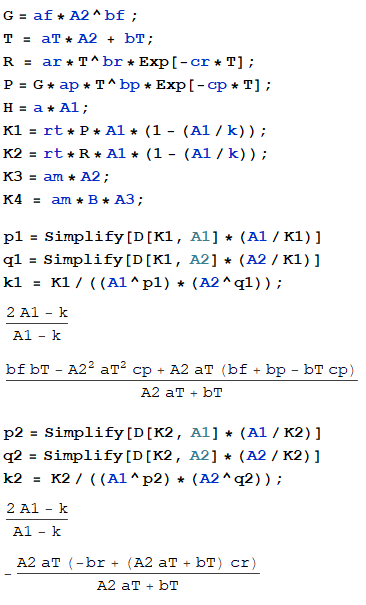}
      \caption{Mathematica codes.}  \label{fig:codes}
\end{figure}

\end{document}